\documentclass{amsart}
\usepackage{graphicx}
\usepackage{xcolor}
\usepackage{ulem}
\usepackage[colorlinks=true,citecolor=black,linkcolor=black,urlcolor=blue]{hyperref}

\newtheorem{theorem}{Theorem}
\newtheorem{lemma}[theorem]{Lemma}
\newtheorem{proposition}[theorem]{Proposition}
\newtheorem{corollary}[theorem]{Corollary}
\newtheorem{conjecture}[theorem]{Conjecture}

\newtheorem{question}[theorem]{Question}

\theoremstyle{remark}

\newtheorem{example}{Example}
\newcommand{\bExample}{\begin{example}}
\newcommand{\eExample}{\end{example}}

\newcommand{\bEnum}{\begin{enumerate}}
\newcommand{\eEnum}{\end{enumerate}}
\newcommand{\bItemize}{\begin{itemize}}
\newcommand{\eItemize}{\end{itemize}}

\newcommand{\R}{\mathbb R}

\newcommand{\st}{: \;}

\newcommand{\ty}{\nabla\mathrm{Y}}
\newcommand{\yt}{\mathrm{Y}\nabla}
\newcommand{\PP}{\mathcal P}

\newcounter{arb}

\title{Three thousand obstructions to knotless embedding}
\author{Thomas W. Mattman}
\address{Department of Mathematics \& Statistics; CSU, Chico;
Chico, California 95929-0525}
\email{TMattman@CSUChico.edu}
\thanks{The first author was supported in part by the American Mathematical
Society and the Sloan Foundation (Grant G-2025-25245) while the author
participated in the LATTICE 2026 program hosted by the Simons Laufer Mathematical
Sciences Institute in Berkeley, California.}

\author{Andrei Pavelescu}
\address{Department of Mathematics \& Statistics; University of South Alabama;
Mobile, Alabama 36688-0002}
\email{andreipavelescu@southalabama.edu}
\thanks{The second author was supported in part by the AMS-Simons Research Enhancement Grant for PUI Faculty and 
 by the American Mathematical
Society and the Sloan Foundation (Grant G-2025-25245) while the author
participated in the LATTICE 2026 program hosted by the Simons Laufer Mathematical
Sciences Institute in Berkeley, California.}

\begin{document}{

\begin{abstract}  
We present a list of 3028 obstructions to knotless embedding. We survey recent work in this area including: 1) A bibliography of
graphs proven to be intrinsically knotted without relying on computers; 2) An updated listing of obstructions in $\ty$ families including two new large families;
3) Connections with the Colin de Verdi\`ere's invariant including a new obstruction with $\mu = 6$;  and 4) Connectivity of obstructions and their
structure near vertices of degree three or four. We address questions raised in earlier work, (re)state several conjectures, and propose new questions.
\end{abstract}

\maketitle

A consequence of the celebrated Graph Minor Theorem of Robertson and Seymour~\cite{RS} is that the list of obstructions to knotless embedding 
is finite. However, bounding the size of this set remains a daunting challenge. In~\cite{FMMNN}, the authors discuss the 264 obstructions known in 2017. 
Here, we present a list of more than three thousand obstructions. Appendix A gives graph6 notation for these 3028 graphs.

The inspiration to prepare this updated list came from
several fruitful discussions during the Spatial Graphs Session of SumTopo 2025 in Mobile, Alabama.

Our examples range in order from 7 through 19 and in size from 21 through 37. 
Our list is complete through order ten, see~\cite{BBFFHL, CG, CMOPRW, FMMNN, KM, MMR}.
By~\cite{BM, FMMNN, KM, LKLO}, it is also complete through size 23.
We say that a graph is {\em intrinsically knotted (IK)} if every spatial embedding contains a non-trivially knotted cycle.
To verify that the graphs in our list are IK, we mostly rely on the algorithm of Miller and Naimi~\cite{MN}.
A Mathematica implementation of the algorithm is available at \cite{N}. As discussed in~\cite{MN}, a graph with no $D_4$-less embedding is IK, 
and we have used the program to verify that each graph in our list has no $D_4$-less embedding.

To show that a graph is minor minimal for the IK property, or MMIK, we must argue that no proper minor is IK. This is immediate for 
the minors that are $2$-apex. For $n \geq 0$, a graph is {\em $n$-apex} if it can be made planar through the deletion of at most $n$ vertices. 
By \cite{BBFFHL,OT}, $2$-apex graphs are not IK (nIK).
For the remaining proper minors, we largely resort to using the algorithm
of Miller and Naimi to check that they are nIK. This usage is built on two assumptions. First, we are relying on the program to find a $D_4$-less
embedding. Aside from potential errors in the code, the program could produce a false positive
as the current implementation only
checks chordless cycles in the graph, rather than the entire cycle space, see~\cite{MN}. A second assumption 
is a positive response to Foisy's question 
(see \cite{MN}).

\begin{question}[Foisy]
\label{que:ID4}
If $G$ has a $D_4$-less embedding, does $G$ have a knotless embedding? \end{question}

To date, every time our program has uncovered a $D_4$-less embedding and we have tried to convert that into a knotless embedding, 
we have succeeded. For example, our classification of the MMIK graphs of order 10 in \cite{KM} includes knotless embeddings of 
eighteen graphs found in this way.

In summary, while we are confident that all graphs in our list are IK, we acknowledge that, for many of them, the argument that each proper minor is nIK 
rests on assumptions that remain to be verified. In spite of this we will refer to the graphs in our list as MMIK graphs or obstructions to 
knotless embedding.

On the other hand, several graphs have been shown to be IK without relying on computer programs and we give a listing of such graphs in
Section~\ref{sec:hand}.


\begin{table}[ht]
\begin{center}
\begin{tabular}{|l|c|c|c|c|} \hline
Germ & $\mu$ & Cousins & IK  & MMIK  \\ \hline
$K_7$ & 6 &  20 & 14 & 14 \\ \hline
$K_{3,3,1,1}$ & 6 & 58 & 58 & 58 \\ \hline
$E_9+e$ & 6 & 110 & 110 & 33 \\ \hline
$G_S$ & 6 & 5 & 5 & 1 \\ \hline
$G_{13,30}$ & 5 & 1 & 1 & 1 \\ \hline
$G_{9,28}$ & 5 & 1609 & 1609 &  1344 \\ \hline
$G_{13,25}$ & 5 & 1,418,605 &  2 & 2 \\ \hline
$G_{10,26}$ & 5 & 4959 &  612 & 82 \\ \hline
$G^1_{12,24}$ & 5 & 23,659 & 1 & 1 \\ \hline
$G^2_{12,24}$ & 5 & 207,426 & 2 & 2 \\ \hline
$G_{10,30}$ & 5 & 8,000,000+ & 1985+ & 1+ \\ \hline
$G_{11,37}$ & 5 & 750,000+ & 114,928+ & 8+ \\ \hline
\end{tabular}
\vspace*{.1in}
\caption{Family counts.} \label{tbl:results}
\end{center}
\end{table}
Many of our graphs occur in $\ty$ families and we update the results of \cite{GMN} for some of the prominent families, see Table~\ref{tbl:results}.
We define these families and discuss the  table entries in Section~\ref{sec:TYfam}. Note that $\mu$ refers to the Colin de Verdi\`ere invariant, 
see~\cite{CdV}.

 It is known that $\mu \geq 5$ for a MMIK graph since a graph with no linkless embedding has $\mu > 4$~\cite{LS} and a graph
with no knotless embedding also fails to have a linkless embedding~\cite{RST}. The MMIK graphs in our list all have $5 \leq \mu \leq 6$,
which suggests the following question.
\begin{question}
\label{ques:mu6}
If $G$ is MMIK, is $\mu(G) \leq 6$?
\end{question}
This would follow from a positive response to the related question below.
\begin{question}[\cite{BM}]
\label{que:3apex}
Is every MMIK graph $3$-apex?
\end{question}
We have verified that the graphs in our list are all $3$-apex.

Van der Holst has proposed a classification of the obstructions to $\mu \leq 5$.
\begin{conjecture}[\cite{vdH}]
\label{con:vdH}
The graphs in the $K_7$ and $K_{3,3,1,1}$ families are the obstructions to $\mu \leq 5$.
\end{conjecture}
Making use of this conjecture, in Section~\ref{sec:mu}, we present a listing of the MMIK graphs with $\mu \geq 6$ through size 24,
including a new example of size 24. Note that all of these graphs have $\mu(G) = 6$.

If a graph is MMIK, then the minimal degree is at least three. 
In Section~\ref{sec:claws}
we discuss the structure of a MMIK graph near a vertex of degree three or four and show that the
connectivity $\kappa$ is between three and nine.
The largest connectivity we have observed is $\kappa = 6$ and we found exactly
two graphs that realize this maximum. 

Our observations about neighborhoods of degree four vertices yield a positive response to our earlier question.

\begin{question}[\cite{KM}]
\label{que:d3MMIK}
If $G$ has a vertex of degree less than three, can an ancestor or descendant of $G$ be an obstruction
for knotless embedding?
\end{question}

As shown in that paper, descendants and parents of $G$ are not obstructions. However, our list includes examples showing that 
grandparents and more distant ancestors may be.

We conclude the paper with some observations about two large families of graphs one including graphs of size 30, and the other consisting of graphs of size 37, the largest size for MMIK graphs that we have observed.

\section{Proofs by hand}
\label{sec:hand}
For most graphs $G$ in our list, the proof that $G$ is IK relies on a computer implementation of the algorithm of
Miller and Naimi~\cite{MN} to show that there is no $D_4$-less embedding. Here we give a bibliography
of graphs that have been shown to be IK without relying on computers. For most of these, we also
give a count of the number of descendants (including the original graph) as these are also IK by the following.
\begin{lemma} If $G$ is IK and $H$ is a descendant of $G$, then $H$ is also IK.
\label{lem:ty}
\end{lemma}

\begin{proof} Sachs~\cite{S} showed that $\yt$ moves preserve nIL. Essentially the same argument
applies to the nIK property. Equivalently, $\ty$ moves preserve IK, so any descendant of an IK graph
is also IK.
\end{proof}

\begin{description}
\item[$K_7$] \cite{CG} 14 descendants.
\item[$K_{3,3,1,1}$] \cite{F} 26 descendants.
\item[$G_{13,30}$] \cite{F2} 1 descendant.
\item[Foisy graphs] Foisy~\cite{F3} presented four graphs $G_{14}$, $H_{15}$, $J_{14}$, and $J'_{14}$ and showed 
they are IK.
\item[$G_{9,28}$] \cite{GMN} 1062 descendants.
\item[$G_{10,26}$] \cite{MNPP} 72 descendants.
\item[$G_{11,23}$] There are two $G_{11,23}$ graphs shown to be MMIK in~\cite{KM}, one with one descendant, the other with two.
\end{description}

\section{$\ty$ families}
\label{sec:TYfam}

In this section we review Table~\ref{tbl:results}. Many of these families were discussed in~\cite{GMN} and we refer the reader to that paper for more details. Compared to the earlier paper, 
we have discovered more MMIK graphs in the $G_{9,28}$ family.
The earlier value of 156 was meant as a lower bound. Moreover, those graphs were shown minor minimal without assuming
a positive response to 
Question~\ref{que:ID4}. In contrast the updated value of 1344 includes additional examples that rely on the conjecture.
So far, 
each family has a unique MMIK graph of minimal order, which we refer to as the {\em germ}. The family described as $G_{14,25}$ in \cite{GMN} is now
labelled by its germ $G_{13,25}$. The graph $G_S$ is an $(11,22)$ graph mentioned in~\cite{FMMNN}. The family of $G_S$ is called the $H_9+e$ family in \cite{KM}.
The graph $G_{13,30}$ was introduced by Foisy~\cite{F2}.
The graphs $G_{10,26}$ and $G_{10,30}$ were found in \cite{MNPP}. Here are edge lists for the remaining three germs. The second, $G^2_{12,24}$, is $4$-regular,
while the first, $G^1_{12,24}$, is not.



$$G^1_{12,24}: {[} (0, 1) , (0, 2) , (0, 3) , (0, 4) , (1, 5) , (1, 6) , (1, 7) , (2, 5) , (2, 8) , $$
$$ (2, 9) , (3, 6) ,  (3, 10) , (3, 11) , (4, 8) , (4, 10) , (4, 11) , (5, 10) ,$$
$$ (5, 11) , (6, 8) , (6, 9) , (7, 8) , (7, 10) , (7, 11) , (9, 10) {]} $$

$$G^2_{12,24}: {[} (0, 1) , (0, 2) , (0, 3) , (0, 4) , (1, 5) , (1, 6) , (1, 7) , (2, 5) , (2, 8) , $$
$$(2, 9) , (3, 6) , (3, 10) , (3, 11) , (4, 9) , (4, 10) , (4, 11) , (5, 8) , $$
$$(5, 10) , (6, 8) , (6, 9) , (7, 9) , (7, 10) , (7, 11) , (8, 11) {]}$$

$$G_{11,37}: {[}(0, 2), (0, 4), (0, 5), (0, 6), (0, 8), (0, 9), (0, 10),$$
$$(1, 3), (1, 5), (1, 6), (1, 7), (1, 8), (1, 9), (1, 10), (2, 4),$$
$$(2, 6), (2, 7), (2, 8), (2, 9), (2, 10), (3, 5), (3, 7), (3, 8),$$
$$(4, 6), (4, 9), (5, 7), (5, 8), (5, 9), (5, 10), (6, 8), (6, 9),$$
$$(6, 10), (7, 8), (7, 9), (7, 10), (8, 10), (9, 10){]}$$

In the family of $G^1_{12,24}$ all but twenty of the 
graphs are 2-apex and of the twenty, all but one has a 2-apex descendant. Therefore, all but one of the graphs is IK. The remaining graph, $G^1_{12,24}$ itself, is MMIK and the only IK graph in the family. 

The situation for the $G^2_{12,24}$ family is similar.
All but 21 are 2-apex and of those 21, all but two have
a $2$-apex descendant. The remaining two graphs,
$G^2_{12,24}$ and its child, are MMIK.

In the table, $\mu$ refers to the Colin de Verdi\`ere invariant, which we discuss further in the next section. 
Here we briefly note that all MMIK graphs $G$ in our list have $5 \leq \mu(G) \leq 6$. In this range, $\mu$ is invariant 
under $\ty$ and $\yt$ moves~\cite{HLS}, meaning all graphs in a family have the same $\mu$ value.

In the table, for each germ we list the $\mu$ value,  the number of cousins in the family, and the number of IK and MMIK graphs in the family. 
Note that the MMIK number is a count of the graphs that appear in our list, some of which are included based on our assumption about Question~\ref{que:ID4}.
For the germs $G_{10,30}$ and $G_{11,37}$ we have only a lower bound on the number of graphs and (MM)IK graphs  in the family.
We say more about these families in the final
section of this paper.


\section{The Colin de Verdi\`ere invariant}
\label{sec:mu}

In this section we briefly summarize properties of the Colin de Verdi\`re invariant $\mu$ and provide the list of MMIK graphs $G$ through size 24 with $\mu(G) \geq 6$.
As mentioned in the introduction, a MMIK graph $G$ has $\mu(G) \geq 5$. All the MMIK graphs $G$ in our list are $3$-apex and it is straightforward to argue (see \cite[Lemma 5.4]{MNPP}), this implies $\mu(G) \leq 6$. Making use of van der Holst's Conjecture~\ref{con:vdH} (stated in the Introduction), the MMIK graphs with $\mu \geq 6$ are precisely those that have 
a minor in the $K_7$ or $K_{3,3,1,1}$ families. All 58 graphs in the $K_{3,3,1,1}$ family are MMIK as are 14 of the 20 graphs in the $K_7$ family~\cite{GMN}.
Therefore, any further examples of MMIK graphs with $\mu \geq 6$ have minors among the six nIK graphs in the $K_7$ family. This includes the 34
MMIK graphs in the  $E_9+e$ and $G_S$ families. The classification of MMIK graphs of size at most 22 is summarized in \cite{FMMNN} and we conclude
that all such graphs have $\mu=6$. There are three MMIK graphs of order 23~\cite{KM}, and they do not have a minor in the $K_7$ family. By van der Holst's conjecture these three have $\mu = 5$.

For the remainder of this section, we assume van der Holst's conjecture, Conjecture~\ref{con:vdH}. A MMIK graph $G$ with $\mu(G) \geq 6$ and size $\|G\| \geq 24$ has one of the six nIK graphs in the $K_7$ family as a minor. That is,  $G$ is an {\em expansion} of a nIK $K_7$ family graph.
As discussed in~\cite{KM}, the nIK graphs of order 22 formed by adding an edge to a nIK $K_7$ family graph are the 29 nIK graphs in the $H_8+e$ family and the remaining nIK expansions of order 22 have a vertex of degree at most two. In total, we found that there are 62 nIK expansions of order 22 and they 
are listed in Appendix B. A computer search of the order 23 expansions of these 62 graphs uncovered 955 nIK expansions of the six nIK graphs in the $K_7$ family,
which are also listed in Appendix B.
Thus, a MMIK graph $G$ of order 24 and $\mu(G) \geq 6$ is an expansion of one of those 955. 
A computer search indicates that there is exactly one 
size 24 MMIK graph that is an expansion of one of the 955. Here is an edge list of that graph, $G_{13,24}$.

$$G_{13,24} :
{[}(0, 2), (0, 3), (0, 4), (0, 5), (1, 5), (1, 9), (1, 10), (1, 12), (2, 5),$$
$$ (2, 6), (2, 7), (3, 9), (3, 11), (3, 12), (4, 6), (4, 8), (4, 10), $$
$$(5, 8), (6, 8), (6, 9), (7, 10), (7, 11), (7, 12), (8, 11)]$$

Contract the edges $(1,12)$, $(4,10)$, and $(6,9)$. Use, in each case, the smaller endpoint as the label of the vertex formed by contraction. 
This means the resulting graph on 10 vertices has no vertices labelled 9 or 10.
It is a nIK graph in the
$K_7$ family with the following edges. 

$$N_{10}: {[}(0, 2), (0, 3), (0, 4), (0, 5), (1, 3), (1, 4), (1, 5), (1, 6), $$
$$(1, 7), (2, 5), (2, 6), (2, 7), (3, 6), (3, 11), (4, 6),$$
$$ (4, 7), (4, 8), (5, 8), (6, 8), (7, 11), (8, 11){]}$$

Since $G_{13,24}$ has a minor in the $K_7$ family and $\mu$ is monotone under minors~\cite{CdV} $\mu(G_{13,24}) \geq 6$. On the other hand,
$G_{13,24}$ is $3$-apex, so $\mu(G_{13,24}) \leq 6$ (see Lemma 5.4 of \cite{MNPP}). Thus, $\mu(G_{13,24}) = 6$ and, using Conjecture~\ref{con:vdH},
appears to be the only graph of size 24 with $\mu = 6$. It is also the only $\mu = 6$ MMIK graph with size greater than 22 that we know of.

\section{Claws in MMIK graphs}
\label{sec:claws}




In this section we discuss the structure of a MMIK graph near a vertex of degree three or four as well as the connectivity of these graphs.
We begin by observing that the neighborhood of a degree three vertex induces a $K_{3,1}$ or claw subgraph.
The following result refers to properties, such as IK and non-planar,
that are preserved by the $\ty$ move and by contracting an edge on a degree two vertex.

\begin{lemma}
\label{lem-mmikclaw}
Let $\PP$ be any graph property that is preserved by the $\ty$ move
and by contracting an edge on a degree two vertex.
Then in every graph that is minor minimal with respect to property P,
every vertex of degree 3 has a claw neighborhood.
\end{lemma}
\begin{proof}
Suppose $G$ has property $\PP$,
$x$ has degree 3 in $G$,
its neighbors are $a, b, c$,
and $ab$ is an edge of $G$.
Do a $\ty$ move on $xab$.
Let $y$ denote the center of the new Y.
Now contract $xy$.
The graph we get is isomorphic to $G - ab$
and has property $\PP$.
So $G$ is not minor minimal with respect to P.
\end{proof}

\begin{corollary}
\label{cor-mmikclaw}
Vertices of degree 3 in MMIK graphs
are not in 3-cycles.
\end{corollary}

Equivalently, since MMIK graphs have minimum degree three, every vertex in a 3-cycle has degree at least four.

\begin{proposition}
\label{prop-mmik3cut}
If $G$ is a MMIK graph and $\{a, b, c\}$ is a 3-cut in $G$,
then $G$ has a vertex $x$ such that
the subgraph of $G$ induced by $\{x,a,b,c\}$ consists of
the edges $xa, xb, xc$ only.
\end{proposition}

\begin{proof}
Let $G_1, G_2$ be distinct connected components of $G -a-b-c$.
If, for some $i \in \{1,2\}$, the graph $G_i$ has only one vertex,
then we are done  by Lemma~\ref{lem-mmikclaw}.
So assume each $G_i$ has at least two vertices; we will show $G$ is nIK.
Let $H_i$ be the proper minor of $G$ obtained by contracting $G_i$ to one vertex, $x_i$.
Since $G$ is MMIK, $H_i$ has a knotless embedding, $\eta_i$, in $\R^3$.
Isotope $\eta_1$ so that it lies in $\R^3_+ := \{(x,y,z) \st z \ge 0\}$
and intersects the $xy$-plane in precisely $\{x_1a, x_1b, x_1c\}$.
Similarly, isotope $\eta_2$ to lie in $\R^3_- := \{(x,y,z) \st z \le 0\}$,
but with the added condition that $x_2 a, x_2 b, x_2 c$ coincide, respectively, 
with   $x_1 a, x_1 b, x_1 c$.
Let $\eta = \eta_1 \cup \eta_2$,
with $x = x_1 = x_2$.
Then $\eta - x =G$.
So, to show $G$ is nIK, it suffices to show $\eta$ is knotless.

Let $C$ be a cycle in $\eta$.
If $C \subset \eta_i$ for some $i$,
then $C$ is a trivial knot since $\eta_i$ is knotless.
So assume $C$ has 
at least one vertex above, and at least one vertex below, the $xy$-plane.

Case 1. 
$x \not \in C$.
Then we can assume 
$C= (a, v_1, \cdots, v_s, b, w_1, \cdots, w_t, a)$
for some $s \ge 1$ and $t \ge 1$, 
with $v_1, \cdots, v_s$ above the $xy$-plane
and $w_1, \cdots, w_t$ below the $xy$-plane.
The cycles 
$(a, v_1, \cdots, v_s, b, x, a)$ and
$(b, w_1, \cdots, w_t, a, x, b)$
are trivial knots since they lie in $\eta_1$ and $\eta_2$, respectively.
These two cycles intersect in the path $axb$,
so their connected sum, which is $C$, is a trivial knot.

Case 2. 
$x \in C$.
Then we can assume 
$C= (a, v_1, \cdots, v_s, b, x, c, w_1, \cdots, w_t, a)$
for some $s \ge 1$ and $t \ge 1$, 
with $v_1, \cdots, v_s$ above the $xy$-plane
and $w_1, \cdots, w_t$ below the $xy$-plane.
The cycles 
$(a, v_1, \cdots, v_s, b, x, a)$ and
$(c, w_1, \cdots, w_t, a, x, c)$
are trivial knots since they lie in $\eta_1$ and $\eta_2$, respectively.
These two cycles intersect in the edge $ax$,
so their connected sum, which is $C$, is a trivial knot.
\end{proof}

In contrast,  for degree four vertices,
all possible induced subgraphs arise:
If $H$ is an order five graph that includes a vertex of degree
4, then $H$ occurs as the induced subgraph on the neighborhood
of a degree 4 vertex in some MMIK graph.

Of particular note are the graphs where the neighborhood of a degree 4 vertex is a butterfly or bowtie graph, 
formed by identifying two triangles at the degree four vertex. Performing the two $\ty$ moves
results in a grandchild with a degree two vertex. This shows that the grandchild is a graph with a degree two vertex that
has a MMIK grandparent yielding a positive response to Question~\ref{que:d3MMIK}.


Let $\kappa(G)$ denote the vertex connectivity of graph $G$.
\begin{lemma} If $G$ is MMIK, then $3 \leq \kappa(G) \leq 9$.
\end{lemma}

\begin{proof} 
We begin with the lower bound.
If $G$ is not connected or 1-connected, then $G$ is not MMIK. Suppose $G$ is 2-connected with cut set ${a,b}$. Let $H_1$ and
$H_2$ be subgraphs of $G$ with $V(H_1) \cap V(H_2) = \{a,b\}$ and $G = H_1 \cup H_2$.
If $ab$ is an edge of $G$,
place an unknotted embedding of $H_1$ above the $xy$-plane and $H_2$ below the $xy$-plane with
the edge $ab$ in the plane. Then a cycle in $G$ is either in an $H_i$ and, therefore unknotted, or else
is the connected sum along $ab$ of two unknotted cycles in $H_1$ and $H_2$. This is a knotless embedding
of $G$.

If $ab$ is not an edge, we can contract the half of the graph to produce that edge:
That is, by contracting $H_{i+1}$ to the edge $ab$, we see that $H_i + ab$ is nIK.
As in the previous case, take knotless embedding of each $H_i + ab$ ($i = 1,2$) and join them along the edge $ab$.
Again, each cycle in $G$ is either in an $H_i+ab$ or the connected sum of two unknots.
This shows that $\kappa(G) \geq 3$.

For the upper bound, 
Mader~\cite{Md} showed that if $|G| \geq 7$ and $\|G \| \geq 5|G|-14$, then $G$ has a $K_7$ minor
and is, therefore, IK. It follows that if $\delta(G) \geq 10$, then $G$ is IK but not MMIK. The upper bound follows as
 $\kappa(G) \leq \delta(G)$.
 \end{proof}
 
\begin{table}[ht]
\begin{center}
\begin{tabular}{|l|c|c|c|c|} \hline
$\kappa(G)$ & 3 & 4 & 5 & 6 \\ \hline
Count & 2739 & 257 & 30 & 2 \\ \hline
\end{tabular}
\vspace*{.1in}
\caption{Connectivity distribution.} \label{tbl:conn}
\end{center}
\end{table}

Table~\ref{tbl:conn} shows the distribution of our MMIK examples by connectivity. 
The largest connectivity observed is realized by $K_7$ and
 $G_{9,28}$ (see \cite{GMN}), both with $\kappa(G) = 6$. We can describe $G_{9,28}$ as the complement
 of $C_7 \sqcup K_2$, the disjoint union of a seven cycle and $K_2$. This is reminiscent 
 of the situation for apex obstructions where the largest connectivity is $\kappa(G) = 5$, 
 and, to date, the only known obstructions that realize this maximum are $K_6$ and 
 the J\/orgensen graph, which is the complement of $C_6 \sqcup K_2$, see~\cite{PPP}. Note that graphs
 with complement $C_n \sqcup K_2$ are identified as special examples realizing
 $\mu(G) < n-1$, see \cite[Theorem 5.5 and Figure 4]{HLS}.
 
We conclude this section with an observation about nIK graphs that have a $\ty$ move to an IK graph.

\begin{lemma} Let $G$ be nIK with a $K_3$ subgraph.
The graph $G$ admits no knotless embedding with the $K_3$ panelled if and only if
a $\ty$ move on the $K_3$ gives an IK graph.
\end{lemma}

\begin{proof}
Let $H$ denote the graph that results from a $\ty$ move on the $K_3$.
If there were such a knotless embedding, we could realize the $\ty$ move in the panel to get a knotless
embedding of $H$. Conversely, if $H$ is nIK, then, in a knotless embedding of
$H$, we can make a panel for the $K_3$ inside a neighborhood of the $Y$ in $H$ to obtain a knotless
embedding of $G$ with the $K_3$ paneled.
\end{proof}

\section{Two large families}

In \cite{GMN} the authors began an exploration of MMIK graphs in 
$\ty$ families, including the $G_{13,25}$ family. At the time
we saw that were more then 0.6 million graphs in the family 
without reaching a precise upper bound. In the meantime we have 
determined the exact number, which is more than 1.4 million. 
On the other hand, we have uncovered an even larger family.

The graph $G_{10,30}$ was introduced in \cite{MNPP}. To date we have not been able to bound the number of
graphs in the $G_{10,30}$ family, but know it is at least eight million.
In his thesis, Herron~\cite{H} proved this graph is minor minimal for IK and we
summarize his argument here. The strategy is to show that every 
graph $H$ obtained by deleting or contracting an edge is nIK.
All edge-contraction minors are 2-apex. All but seven of the edge-deletion subgraphs are 2-apex. Six of the remaining seven have a 2-apex descendant and are nIK by Lemma~\ref{lem:ty}.
Herron's thesis includes a figure realizing a knotless embedding of
the final edge-deletion subgraph, which concludes the argument.

The highly-symmetric graph $G_{11,37}$ has an interesting structure. (See Section~\ref{sec:TYfam} for an edge list of this graph.) It is obtained by adding a vertex of degree eight to a maximal linklessly embedable graph of order ten. As such, its $\mu$-invariant is at most five. Given that $G_{11,37}$ is IK, it follows that its $\mu$-invariant is exactly five. The family of $G_{11,37}$ includes (at least) eight MMIK graphs of size 37, the largest size that occurs in our list of three thousand. 
The 100,000+ IK graphs that we know of in this family constitute a complete list of its descendants. 
We noticed that the descendants of
$G_{11,37}$ form a tree, which appears to be 
quite exceptional. 
For example, \cite{GMN} includes (parts of) graphs for three $\ty$ families
all of which include an abundance of cycles.

A preliminary investigation shows that among the simple minors (ie by contracting or deleting a single edge) of 
$G_{11,37}$, all but eight have 2-apex descendants and are nIK for this reason. We leave it as an open question to find knotless embeddings for the remaining eight graphs.

\section*{Acknowledgements}
We owe a large debt to Ramin Naimi and Elena Pavelescu for many detailed discussions about ideas in this paper, especially the contents of Section 4. Madeline Potter wrote code and performed searches that were 
crucial in forming our list of three thousand MMIK graphs.

}\end{document}